\documentclass{amsart}
\usepackage{amsmath,amssymb,amsfonts}
\usepackage{graphicx,epsfig}
\usepackage{amsfonts}
\usepackage{mathrsfs}
\usepackage{amssymb}
\usepackage{amsmath}
\usepackage{amsthm}
\usepackage{hyperref}
\usepackage{lscape}
\usepackage{lscape,graphicx}

\numberwithin{equation}{section}

\theoremstyle{plain}
\newtheorem{theorem}{Theorem}[section]

\newcommand{\G}{\Gamma}
\newcommand{\h}{\frac{1}{2}}

\newcommand{\al}{\alpha}

\begin{document}
\title{A note on a generalization of two well-known combinatorial identities via a hypergeometrric series approach}
\author{Arjun K. Rathie, Insuk Kim$^*$  and  Richard B. Paris   }
\maketitle

\begin{abstract}
In this note, we aim to provide generalizations of 
(i) Knuth's old sum (or Reed Dawson identity) and 
(ii) Riordan's identity 
using a hypergeometric series approach. 
\end{abstract}

\renewcommand{\thefootnote}{}

\footnotetext{* Corresponding author}

\section{Introduction and Results required}

We start with the following well-known combinatorial sums known as Knuth's old sums \cite {gre}, or alternatively as the  Reed Dawson identities viz.
\begin{equation}\label{1.1}
\sum_{k=0}^{2\nu} (-1)^{k}\, \binom{2\nu}{k}\, 2^{-k}\, \binom{2k}{k}=2^{-2\nu}\binom{2\nu}{\nu}
\end{equation}
and 
\begin{equation}\label{1.2}
\sum_{k=0}^{2\nu+1} (-1)^{k} \,\binom{2\nu+1}{k} \,2^{-k} \,\binom{2k}{k}=0.
\end{equation}
It is of interest to mention that Reed Dawson presented the above identities in a private communication to Riordan who recorded them in his well-known book \cite[pp.~71]{rio}.

Several different proofs of the above sums have been given in the literature; see the survey paper by Prodinger \cite{pro1}. 
Jonassen $\&$ Knuth \cite{jon} gave an elementary demonstration using a recursion for the binomial coefficients. Gessel \cite{gre} expressed the binomial coefficients as coefficients in appropriate generating functions. Rousseau \cite{jon} showed that the sums could be expressed in terms of the constant coefficient in the expansion of $(x^2+x^{-2})^n$  and Prodinger \cite{pro2} employed the Euler transformation.

In 1974, Andrews \cite[pp.~478]{and} established the above sums by employing the Gauss second summation theorem \cite{rai} given by
\begin{equation}\label{1.3}
		{}_{2}F_{1}\left[ \begin{matrix}\,\, a,\quad\quad b\\ \frac{1}{2}( a+b+1)\end{matrix}; \frac{1}{2}\right]
				=  \frac{\Gamma(\frac{1}{2})\,\Gamma(\h a+\h b+\h)}{\Gamma(\h a+\h) \,\Gamma(\h b+\h)}.
	\end{equation}
In 2004, Choi {\it et al.} \cite{cho} utilized the following terminating hypergeometric identities recorded, for example in \cite[pp.~126, 127]{rai},

\begin{equation}\label{1.4}
		{}_{2}F_{1}\left[ \begin{matrix}\,\, -2n,\,\,\,\,\al\\ 2\al\end{matrix}\,;\, 2\right]
				=  \frac{(\h)_n}{(\al+\h)_n}
\end{equation}
and 
\begin{equation}\label{1.5}
		{}_{2}F_{1}\left[ \begin{matrix}\,\, -2n-1,\,\,\,\,\al\\ 2\al\end{matrix}\,;\, 2\right]
				= 0,
\end{equation}
where $(\al)_n$ denotes the Pochhammer symbol  (or the rising factorial) for any complex number $\al(\not= 0)$ defined by 
\begin{equation*}
(\al)_{n}=\begin{cases} \al (\al +1)\ldots(\al+n-1), \;  & (n \in \mathbb{N}) \\ 
                                      1 ,  & \; (n=0) \end{cases}. 
\end{equation*}

Also, the following well-known combinatorial identities established by Riordan \cite{rio} are seen to be closely related to \eqref{1.1} and \eqref{1.2} viz.
\begin{equation}\label{1.6}
\sum_{k=0}^{2\nu} (-1)^{k}\, \binom{2\nu+1}{k+1}\, 2^{-k}\, \binom{2k}{k}=2^{-2\nu}\,(2\nu +1)\binom{2\nu}{\nu}
\end{equation}
and 
\begin{equation}\label{1.7}
\sum_{k=0}^{2\nu+1} (-1)^{k} \,\binom{2\nu+2}{k+1} \,2^{-k} \,\binom{2k}{k}=2^{-2\nu-1}\,(\nu +1)\binom{2\nu}{\nu}.
\end{equation}
Riordan \cite{rio} established \eqref{1.6} and \eqref{1.7} by the method of inverse relations.

Very recently, generalizations of the hypergeometric identities \eqref{1.4} and \eqref{1.5} were given by Kim {\it et al.} \cite{kim}, written in the following form:
\begin{align}\label{1.8}
		{}_{2}F_{1}&\left[ \begin{matrix}\,\, -2n,\,\,\,\,\al\\ 2\al+i\end{matrix}\,;\, 2\right]
				=  \frac{2^{-2\al -i}\,\G(\al)\,\G(1-\al)}{\G(\al +i)\,\G(1-2\al -i)}\\
&\quad\times \sum_{r=0}^{i} (-1)^{r}\,\binom{i}{r}\,\frac{\G(\h -\al-\h i+\frac{1}{2}r)\,(\h+\h i-\h r)_{n}}{\G(\h-\h i+\h r)\,(\al+\h+\h i-\h r)_{n}}\notag
\end{align}
and 
\begin{align}\label{1.9}
		{}_{2}F_{1}&\left[ \begin{matrix}\,\, -2n-1,\,\,\,\,\al\\ 2\al+i\end{matrix}\,;\, 2\right]
				= - \frac{2^{-2\al -i}\,\G(\al)\G(1-\al)}{\G(\al +i)\,\G(1-2\al -i)}\\
&\quad\times \sum_{r=0}^{i} (-1)^{r}\,\binom{i}{r}\frac{\G( -\al-\h i+\h r)\,(1+\h i-\h r)_{n}}{\G(-\h i+\h r)\,(\al+1+\h i-\h r)_{n}},\notag
\end{align}
which are valid for $i\in \mathbb{N}_{0}$.

In this note, we aim to provide generalizations of Knuth's old sums (or Reed Dawson identities) \eqref{1.1} and \eqref{1.2} and  Riordan's identities \eqref{1.6} and \eqref{1.7} in the most general form for any $i\in \mathbb{N}_{0}$. In order to obtain the results in the most general form for any $i\in \mathbb{N}_{0}$, we have to construct two master formulas. The results are established with the help of \eqref{1.8} and \eqref{1.9}.
In Section 3 we present cases of our general result that correspond to Knuth's old sums and Riordan's identities, together with some interesting new results.

\section{Generalizations} 

The generalizations of Knuth's old sums (or Reed Dawson identities) are given in the following theorem:
\begin{theorem}  For $i\in \mathbb{{N}}_{0}$, the following results hold true.
\begin{align}\label{2.1}
&\sum_{k=0}^{2\nu} (-1)^{k}\, \binom{2\nu+i}{k+i}\, 2^{-k}\, \binom{2k}{k}\\
&=\pi(2\nu+1)_{i}\,\frac{2^{2i}\, i\,!}{(2i)!}\,\sum_{r=0}^{i}\frac{2^{-r}\,\binom{i}{r}\, (\h+\h(i-r))_{\nu}}{(i-r)!\,\G^{2}(\h+\h(r-i))\,(1+\h(i-r))_{\nu}}\notag
\end{align}
and
\begin{align}\label{2.2}
&\sum_{k=0}^{2\nu+1} (-1)^{k}\, \binom{2\nu+1+i}{k+i}\, 2^{-k}\, \binom{2k}{k}\\
&=2\pi(2\nu+2)_{i}\,\frac{2^{2i}\, i\,!}{(2i)!}\,\sum_{r=0}^{i}\frac{2^{-r}\,\binom{i}{r}\, (1+\h(i-r))_{\nu}}{(i-r+1)!\,\G^{2}(\h(r-i))\,(\frac{3}{2}+\h(i-r))_{\nu}}.\notag
\end{align}

\end{theorem}

\begin{proof} 
The proof of the identities \eqref{2.1} and \eqref{2.2}  is straightforward.  For this, let us consider the sum for $i\in \mathbb{{N}}_{0}$:
$$S=\sum_{k=0}^{n} (-1)^{k}\, \binom{n+i}{k+i}\, 2^{-k}\, \binom{2k}{k}.$$
Making use of the identities
\begin{equation*}
(n-k)!= \frac{(-1)^{k}\,n!}{(-n)_{k}},\qquad 
\G(2z)=\frac{2^{2z-1}\,\G(z)\,\G(z+\h)}{\sqrt{\pi}},
\end{equation*}
we have after some simplification,
\begin{align}\label{2.3}
S&=\frac{\G(n+1+i) }{\G(1+i)\,\G(n+1)} \sum_{k=0}^{n}  \frac{(-n)_{k}\,(\h)_{k}\,2^{-k}}{(1+i)_{k}\,k\,!}\\
&=\frac{\G(n+1+i) }{\G(1+i)\,\G(n+1)} \,{}_{2}F_{1}   \left[ \begin{array}[c]{ccccc} -n,\quad \h\\ 1+i\end{array};2 \right].\notag
\end{align}

Now for $n=2\nu$ (even) and $n=2\nu+1$ (odd), the $_2F_1$ appearing in \eqref{2.3}can be evaluated with the help of the known results \eqref{1.8} and \eqref{1.9} and  after some simplification, we easily arrive at the results \eqref{2.1} and \eqref{2.2} asserted by the theorem. This completes the proof of \eqref{2.1} and \eqref{2.2}.
\end{proof}

In the next section, we shall mention the known summations presented in Section 1 as well as two new cases of our main findings.

\vspace{0.5cm}

\section{Corollary} 

1. If we take $i=0$ in Theorem 2.1, we obtain
\begin{equation}\label{3.1}
\sum_{k=0}^{2\nu} (-1)^{k}\, \binom{2\nu}{k}\, 2^{-k}\, \binom{2k}{k}=\frac{(\h)_{\nu}}{(1)_{\nu}}
\end{equation}
and 
\begin{equation}\label{3.2}
\sum_{k=0}^{2\nu+1} (-1)^{k} \,\binom{2\nu+1}{k} \,2^{-k} \,\binom{2k}{k}=0,
\end{equation}
which are equivalent to Knuth's old sums (or the Reed Dawson identities) \eqref{1.1} and \eqref{1.2}.

\vspace{0.5cm}
2.  If we take $i=1$ in Theorem 2.1, we obtain
\begin{equation}\label{3.3}
\sum_{k=0}^{2\nu} (-1)^{k}\, \binom{2\nu+1}{k+1}\, 2^{-k}\, \binom{2k}{k}=(2\nu+1)\,\frac{(\h)_{\nu}}{(1)_{\nu}}
\end{equation}
and 
\begin{equation}\label{3.4}
\sum_{k=0}^{2\nu+1} (-1)^{k} \,\binom{2\nu+2}{k+1} \,2^{-k} \,\binom{2k}{k}=(\nu+1)\,\frac{(\frac{3}{2})_{\nu}}{(2)_{\nu}},
\end{equation}
which are equivalent to the Riordan identities \eqref{1.6} and \eqref{1.7}.

\vspace{0.5cm}
3.  Finally, if we take $i=2$ and $i=3$ in Theorem 2.1, we find the following new results.
\begin{equation}\label{3.5}
\sum_{k=0}^{2\nu} (-1)^{k}\, \binom{2\nu+2}{k+2}\, 2^{-k}\, \binom{2k}{k}=\frac{1}{3}(4\nu+3)\,\frac{(\frac{3}{2})_{\nu}}{(1)_{\nu}}
\end{equation}
\begin{equation}\label{3.6}
\sum_{k=0}^{2\nu+1} (-1)^{k} \,\binom{2\nu+3}{k+2} \,2^{-k} \,\binom{2k}{k}=2\,\frac{(\frac{5}{2})_{\nu}}{(1)_{\nu}}.
\end{equation}
and
\begin{equation}\label{3.7}
\sum_{k=0}^{2\nu} (-1)^{k}\, \binom{2\nu+3}{k+3}\, 2^{-k}\, \binom{2k}{k}=\frac{1}{5}(8\nu+5)\,\frac{(\frac{5}{2})_{\nu}}{(1)_{\nu}}
\end{equation}
\begin{equation}\label{3.8}
\sum_{k=0}^{2\nu+1} (-1)^{k} \,\binom{2\nu+4}{k+3} \,2^{-k} \,\binom{2k}{k}=\frac{1}{5}(8\nu+15)\,\frac{(\frac{5}{2})_{\nu}}{(1)_{\nu}}.
\end{equation}
Similarly, other results can be obtained.

\vspace{0.5cm}\noindent
{\bf Acknowledgement} 

The research work of Insuk Kim is supported by Wonkwang University in 2020.

\vspace{0.5cm}\noindent
{\bf Authors' contributions} 

All authors contributed equally to writing of this paper. All authors read and approved the final manuscript.

\vspace{0.5cm}\noindent
{\bf Competing interest} 

The authors declare that they have no competing interests.

\vskip .5cm
\noindent {\bf Authors' affiliations}

Arjun K. Rathie: Department of Mathematics, Vedant College of Engineering and Technology (Rajasthan Technical University), Bundi, 323021, Rajasthan, India\\
E-mail : {arjunkumarrathie@gmail.com}\\

Insuk Kim: Department of Mathematics Education, Wonkwang University, Iksan, 570-749, Republic of Korea\\
E-mail: {iki@wku.ac.kr} \\

Richard B. Paris: Division of Computing and Mathematics, University of Abertay, Dundee DD1 1HG, UK\\
E-mail: {r.paris@abertay.ac.uk} \\

\end{document}